\documentclass{amsart}
\usepackage{shortsalch}
\usepackage{xy}
\usepackage{mathtools}
\xyoption{all}
\usepackage{cleveref}
\title[The Steenrod algebra: self-injective, and not self-injective.]{The Steenrod algebra is self-injective, and the Steenrod algebra is not self-injective.}
\author{A. Salch}
\begin{document}
\begin{abstract}
It is well-known that the Steenrod algebra $A$ is self-injective as a graded ring. We make the observation that simply changing the grading on $A$ can make it cease to be self-injective. We see also that $A$ is {\em not} self-injective as an {\em ungraded} ring. 

These observations follow from the failure of certain coproducts of injective $A$-modules to be injective. Hence it is natural to ask: which coproducts of graded-injective modules, over a general graded ring, remain graded-injective? We give a complete solution to that question by proving a graded generalization of Carl Faith's characterization of $\Sigma$-injective modules. Specializing again to the Steenrod algebra, we use our graded generalization of Faith's theorem to prove that the covariant embedding of graded $A_*$-comodules into graded $A$-modules preserves injectivity of bounded-above objects, but does not preserve injectivity in general.
\end{abstract}
\maketitle

\section{Self-injectivity of the Steenrod algebra.}

The following is a well-known theorem, originally due to Adams and Margolis \cite{MR294450} at the prime $p=2$, and Moore and Peterson \cite{MR335572} at odd primes $p$:
\begin{theorem}\label{A is self-inj}
  The mod $p$ Steenrod algebra is self-injective. More precisely: the mod $p$ Steenrod algebra $A$, regarded as a free {\em graded} left $A$-module, is injective in the category of {\em graded} left $A$-modules.
\end{theorem}
More generally, any bounded-below free graded left $A$-module is injective in the category of graded left $A$-modules. See Theorem 12 in section 13.3 of \cite{MR738973} for this result, as well as its converse: a bounded-below graded left $A$-module is injective if and only if it is free.

However, using some old results in ring theory, it is also easy to prove the following:
\begin{theorem}\label{A is not self-inj}
The Steenrod algebra is {\em not} self-injective. More precisely: the Steenrod algebra $A$, as a free left $A$-module, is {\em not} injective in the category of {\em ungraded} left $A$-modules.
\end{theorem}
\begin{proof}
Recall that an injective module is said to be {\em $\Sigma$-injective} if every coproduct of copies of that module is also injective. The main theorem of Megibben's paper \cite{MR633266} establishes that all countable injective modules, over any ring, are $\Sigma$-injective. If $A$ were injective in the category of (ungraded) left $A$-modules, then by Megibben's theorem, the coproduct $\coprod_{n\in\mathbb{Z}} A$ would also be an injective left $A$-module. The coproduct $\coprod_{n\in\mathbb{Z}} A$ is the underlying ungraded left $A$-module of the graded left $A$-module $\coprod_{n\in\mathbb{Z}} \Sigma^n A$, which is known to {\em not} be injective in the category of graded $A$-modules: this is a special case of 
%Theorem 2.7 from Moore and Peterson's paper \cite{MR335572} (NO IT ISN'T!)
Proposition 10 in section 13.2 of Margolis's book \cite{MR738973}, which establishes that a free graded left $A$-module cannot be injective in the graded module category unless it is bounded below.

Finally, the functor $U$ from the graded module category to the ungraded module category has the property that, if $UM$ is injective, then $M$ is also injective. This is classical, and holds for any nonnegatively-graded ring: see for example Corollary 3.3.10 of \cite{MR551625}. The argument is now complete: if $A$ were self-injective in the ungraded module category, then the ungraded direct sum $\coprod_{n\in\mathbb{Z}} A$ would also be injective, by Megibben's theorem; hence $\coprod_{n\in\mathbb{Z}} \Sigma^n A$ would be injective in the graded module category, contradicting Margolis's result.
\end{proof}
Despite this paper's tongue-in-cheek title, there is of course no contradiction  implied by Theorems \ref{A is self-inj} and \ref{A is not self-inj} being both true: it is entirely consistent for a graded ring to be self-injective in the graded sense, but not self-injective in the ungraded sense. 

\begin{corollary}
Changing the grading on the Steenrod algebra, or forgetting the grading altogether, can change whether or not the Steenrod algebra is self-injective.
\end{corollary}
\begin{proof}
We have just shown that forgetting the grading on the Steenrod algebra results in a ring which is not self-injective. As a consequence, if we change the grading on $A$ by putting all its elements in degree $0$, then the free $A$-module $A$ is not injective in the resulting graded module category.
\end{proof}

\section{A graded version of Faith's criterion for $\Sigma$-injectivity.}
In the proof of Theorem \ref{A is not self-inj}, it was useful to consider whether a coproduct of copies of an injective (graded or ungraded) module remains injective. An injective module $M$ is called {\em $\Sigma$-injective} (respectively, {\em countably $\Sigma$-injective}) if the direct sum $\coprod_{s\in S}M$ is injective for all sets $S$ (respectively, all countable sets $S$). The universal property of injective modules ensures that injectivity is preserved under {\em products,} but it does not directly imply anything about injectivity of {\em coproducts,} i.e., direct sums. The classical Bass-Papp theorem\footnote{See Theorem 3.46 of \cite{MR1653294} for a textbook treatment.} states that, if $R$ is a ring, then $R$ is left Noetherian if and only if every coproduct of injective left $R$-modules is injective. So over a left Noetherian ring, all injective left modules are $\Sigma$-injective.

The Steenrod algebra is not Noetherian on either side, so the Bass-Papp theorem does not apply. For non-Noetherian rings, the standard tool for determining whether a given injective module is $\Sigma$-injective is the following theorem of Faith, from \cite{MR193107}:
\begin{theorem}
Suppose that $R$ is a ring, and that $M$ is an injective right $R$-module. %Write $\Lambda$ for the $R$-linear endomorphism ring of $M$. 
Then the following conditions are equivalent:
\begin{enumerate}
\item $M$ is $\Sigma$-injective.
\item $M$ is countably $\Sigma$-injective.
\item $R$ satisfies the ascending chain condition on its right ideals which are annihilators of subsets of $M$.
%\item $M$ satisfies the descending chain condition on its left $\Lambda$-submodules which are annihilators of subsets of $R$.
\end{enumerate}
\end{theorem}

Various generalizations of Faith's criterion are known. The most general that the author is aware of is Harada's, from \cite{MR367020}: given a Grothendieck category $\mathcal{C}$ with a generating set $G$ of compact objects and an injective object $Q$ of $\mathcal{C}$, for each left $\hom_{\mathcal{C}}(Q,Q)$-submodule $N$ of $\hom_{\mathcal{C}}(S,Q)$, we can form the limit $\cap_{f\in N} \ker f$, which is a subobject of $S$. Such a subobject of $S$ is called an {\em annihilator ideal of $S$ for $Q$.} Harada proves that $Q$ is $\Sigma$-injective if and only if, for each $S\in G$, the partially-ordered set of annihilator ideals of $S$ for $Q$ satisfies the ascending chain condition. 

While the category of graded $A$-modules is indeed Grothendieck, Harada's generalization of Faith's criterion does not make the fine distinctions that we will need in order to get a good grasp on the graded analogues of $\Sigma$-injectivity. Consider the following result of Margolis\footnote{See sections 13.2 and 13.3 of \cite{MR738973}. The first two parts of Theorem \ref{observation 1} are also consequences of the Moore-Peterson theorem, Theorem 2.7 from \cite{MR335572}, establishing that bounded-below graded $A$-modules are free iff they are gr-injective.}:
\begin{theorem}\label{observation 1}
Let $M$ be a bounded-below free graded module over the Steenrod algebra (with its usual grading). Then $M$ is gr-injective\footnote{Here and from now on, we use a standard piece of terminology from graded ring theory: {\em gr-injective} means ``graded in the injective module category.''}. Furthermore:
\begin{itemize}
\item A direct sum of copies of $M$, without suspensions, remains gr-injective.
\item More generally: a direct sum $\coprod_{s\in S}\Sigma^{d(s)} M$ of suspensions of copies of $M$ remains gr-injective as long as there is a lower bound on the degree $\{ d(s): s\in S\}$ of the suspensions.
\item However, if there is no lower bound on the degrees $\{ d(s): s\in S\}$ of the suspensions, then the direct sum $\coprod_{s\in S}\Sigma^{d(s)} M$ is {\em not} gr-injective.
\end{itemize}
\end{theorem}
It is clear from Theorem \ref{observation 1} that there are several natural but entirely distinct notions of $\Sigma$-injectivity in the graded setting. In the following definition, we give names to these notions of graded $\Sigma$-injectivity.
\begin{definition}
Let $R$ be a graded ring. We say that an gr-injective graded left $R$-module $M$ is:
\begin{itemize}
\item {\em strictly $\Sigma$-injective} if the direct sum $\coprod_{s\in S} M$ is gr-injective for all sets $S$.
\item {\em unboundedly $\Sigma$-injective} if the direct sum $\coprod_{s\in S} \Sigma^{d(s)} M$ is gr-injective for all sets $S$ and all functions $d: S \rightarrow \mathbb{Z}$.
\item {\em bounded-belowly $\Sigma$-injective} (respectively, {\em bounded-abovely $\Sigma$-injective}) if the direct sum $\coprod_{s\in S} \Sigma^{d(s)} M$ is gr-injective for all sets $S$ and all functions $d: S \rightarrow \mathbb{Z}$ such that there is a lower bound (respectively, upper bound) on the values taken by $d$.
\item If $N$ is a set of integers, we say that $M$ is {\em $(\Sigma,N)$-injective} if the direct sum $\coprod_{s\in S} \Sigma^{d(s)} M$ is gr-injective for all sets $S$ and all functions $d: S \rightarrow N$.
\end{itemize}
We have also the countable analogues of the above: for example, we say that $M$ is {\em countably strictly $\Sigma$-injective} if $\coprod_{s\in S} M$ is gr-injective for all countable sets $S$, and so on.
\end{definition}
Harada's theorem can be used to characterize the gr-injective graded modules that are strictly $\Sigma$-injective, or equivalently, $(\Sigma,\{0\})$-injective. There is another precedent in the literature for our study of graded $\Sigma$-injectivity: the paper \cite{MR839577} is about one version of graded $\Sigma$-injectivity, which its authors call ``gr-$\Sigma$-injectivity.'' That notion of ``gr-$\Sigma$-injectivity'' is equivalent to what we call ``unbounded $\Sigma$-injectivity.'' N\v{a}st\v{a}sescu and Raianu prove that, if a graded module is unboundedly $\Sigma$-injective, then its underlying ungraded module is $\Sigma$-injective. This is quite distinct from the notions of bounded-above $\Sigma$-injectivity and bounded-below $\Sigma$-injectivity, which as far as we know have not been studied before. In contrast to unbounded $\Sigma$-injectivity, as a consequence of Corollary \ref{simple cor}, we will see that bounded-above $\Sigma$-injectivity and bounded-below $\Sigma$-injectivity are {\em not} preserved upon forgetting down to the ungraded module category.

The aim of this section is to prove a version of Faith's theorem which gives us a useful necessary and sufficient criterion for each of the various notions of graded $\Sigma$-injectivity.

In order to state our generalization of Faith's theorem, we introduce a graded version of some notation from \cite{MR193107}: given a graded left $R$-module $M$ and a set $X$ of homogeneous elements of $R$, we write $X^{\perp}$ for the set of homogeneous elements $m\in M$ such that $xm=0$ for all $x\in X$. Given a set $Y$ of homogeneous elements of $M$, we write $Y^{\perp}$ for the set of homogeneous elements $r\in R$ such that $ry=0$ for all $y\in Y$. 

We introduce two more simple pieces of notation, and one more piece of terminology: 
\begin{itemize}
\item Given an integer $m$ and a set $N$ of integers, we will write $m + N$ for the set of integers $\{ m+n: n\in N\}$. Similarly, $m - N$ will of course denote the set of integers $\{ m-n: n\in N\}$
\item Given a set $N$ of integers, a graded ring $R$, and a graded left $R$-module $M$, we say that an ascending chain of homogeneous left ideals $I_0\subseteq I_1\subseteq I_2\subseteq\dots$ of $R$ is {\em $M$-annihilator-stable for all degrees in $N$} if there exists some integer $\ell$ such that the submodule inclusions $I_{\ell}^{\perp} \supseteq I_{\ell+1}^{\perp}\supseteq I_{\ell+2}^{\perp}\dots$ are equalities in each grading degree $d$ in $N$, i.e.,
\[ (I_{\ell}^{\perp})^d = (I_{\ell+1}^{\perp})^d = (I_{\ell+2}^{\perp})^d = \dots\]
for all $d\in N$.
\end{itemize}

Now we are ready for the main theorem. It is the graded generalization of the main theorem of Faith's paper \cite{MR193107}. Naturally our proof owes much to Faith's, although some of the ideas in our approach differ from Faith's, 
especially where care concerning the gradings is required.
\begin{theorem}\label{main new faith thm}
Let $R$ be a graded ring, let $N$ be a set of integers, and let $M$ be a gr-injective graded left $R$-module.
Then the following are equivalent:
\begin{enumerate}
\item $M$ is $(\Sigma,N)$-injective.
\item $M$ is countably $(\Sigma,N)$-injective.
\item For each integer $m$, each ascending chain \begin{equation}\label{seq 231} I_0\subseteq I_1 \subseteq I_2 \subseteq \dots\end{equation} of homogeneous left ideals of $R$ is $M$-annihilator-stable for all degrees in $m-N$.
\item $M$ is $(\Sigma,m+N)$-injective for every integer $m$.
\item $M$ is countably $(\Sigma,m+N)$-injective for every integer m.
\end{enumerate}
\end{theorem} 
\begin{proof}\leavevmode
\begin{description}
\item[1 implies 2, 4 implies 5, 4 implies 1, and 5 implies 2]
Immediate.
\item[2 implies 3] 
Suppose that $m$ is an integer, suppose that $M$ is countably $(\Sigma,N)$-injective, and suppose we are given an ascending sequence \eqref{seq 231} of homogeneous left ideals.
Suppose we are given a function $d: \mathbb{N}\rightarrow -m+N$.
For each nonnegative integer $n$, choose a homogeneous element $x_{n}\in I_{n}^{\perp}$ of degree $-d(n)$.
For any given homogeneous element $r\in \bigcup_{n\geq 0}I_n$, there exists some integer $q$ such that $r\in I_q$. Since $I_q\subseteq I_{q+1}\subseteq \dots$, we have also that $I_q^{\perp}\supseteq I_{q+1}^{\perp} \supseteq \dots$, and consequently $ri_n=0$ for all $n>q$. Consequently all but finitely many of the components in 
\[ x(r) := \left( rx_0, rx_1, rx_2, rx_3, \dots \right)\in \prod_{n\geq 0} \Sigma^{d(n)+m}M \]
are zero. That is, $x(r)$ is a homogeneous element of the direct {\em sum} \linebreak $\coprod_{n\geq 0} \Sigma^{d(n)+m}M$, not merely the direct {\em product}. 

The degree of $x(r)$ is equal to $m$ plus the degree of $r$.
Consequently the right $R$-module homomorphism 
\begin{align*} 
 x: \Sigma^m\left( \bigcup_{n\geq 0} I_n \right) &\rightarrow %\coprod_{n\geq 0}\Sigma^{-\left| x_n\right|} M\\ &=
  \coprod_{n\geq 0}\Sigma^{d(n)+m} M\\
 r &\mapsto x(r) \end{align*} respects the grading. Each of the integers $d(n)+m$ is in the set $N$, so by the $(\Sigma,N)$-injectivity of $M$, the graded form of Baer's criterion\footnote{The graded Baer criterion is standard; see e.g. I.2.4 of \cite{nastasescu2011graded}. It is as follows: for any graded ring $R$, a graded $R$-module $M$ is gr-injective if and only if, for every graded left ideal $I$ of $R$ and every diagram 
\[\xymatrix{
 \Sigma^n I \ar[r]\ar[rd] & \Sigma^n R \ar@{-->}[d] \\ & M
}\]
in the category of graded $R$-modules in which the top horizontal map is the canonical inclusion, a map of graded $R$-modules exists which fills in the dotted arrow and makes the diagram commute.} yields a graded right $R$-module homomorphism $g: \Sigma^m R\rightarrow \coprod_{n\geq 0}\Sigma^{d(n)+m} M$ such that $g(r) = x(r)$ for all $r\in \Sigma^m\left(\bigcup_{n\geq 0}I_n\right)$. Since the element $g(1)$ is an element of the direct sum $\coprod_{n\geq 0}\Sigma^{d(n)+m} M$, it must be zero in all but finitely summands. In particular, there must be some integer $\ell$ such that 
\begin{align*} 
 g(1) &= (g_0, g_1, g_2, \dots , g_{\ell}, 0, 0, \dots ) \\
      &\in \coprod_{n\geq 0}\Sigma^{d(n)+m} M,\end{align*}
and consequently
\begin{align*} 
 g(r) &= (rg_0, rg_1, rg_2, \dots , rg_{\ell}, 0, 0, \dots ) \\
      &= (rx_0, rx_1, rx_2, \dots , rx_{\ell}, 0, 0, \dots ) 
\end{align*}
for each $r\in R$. Hence for every $n>\ell$, the element $x_{n}\in M$ is annihilated by $\bigcup_{n\geq 0} I_n$. Hence $x_n \in (\bigcup_{n\geq 0}I_n)^{\perp}$, which is a subset of $I_t^{\perp}$ for every integer $t$.

Now take stock of what we have just shown: we began with an arbitrary sequence $x_1,x_2,x_3,\dots$ of homogeneous elements of $M$ whose degrees are in the set $m-N$, and such that each $x_n$ is in $I_n^{\perp}$. We have just shown that, no matter how these choices are made, there exists some integer $\ell$ such that all the terms $x_{\ell},x_{\ell+1},x_{\ell+2}, \dots$ are in the {\em same} stage $I_{\ell}^{\perp}$. Hence what we have shown is that the sequence $I_1^{\perp}\supseteq I_2^{\perp}\supseteq I_3^{\perp}\supseteq\dots$ is eventually constant in each degree in $m-N$, exactly as we wanted to show.
\item[3 implies 1]
We first claim that condition 3 ensures that, for each homogeneous left ideal $I$ of $R$, there exists a {\em finitely generated}\footnote{To be clear, here and throughout this paper, whenever we say that a homogeneous ideal is ``finitely generated,'' we shall always mean that it has a finite set of {\em homogeneous} generators.} homogeneous left ideal $I_1$ of $R$ contained in $I$ such that $(I^{\perp})^d = (I_1^{\perp})^d$ for every degree $d$ in the set $m-N$. The proof is as follows: consider the collection $\Ideals(M,I,m-N)$ of all finitely generated homogeneous left ideals of $R$ contained in $I$. Preorder this collection by letting $J_1\geq J_2$ if and only if $(J_1^{\perp})^d\subseteq (J_2^{\perp})^d$  for all $d$ in $m-N$. Consequently the set of elements of $\Ideals(M,I,m-N)$ depends only on $I$, not on $M$, not on $m$, and not on $N$. However, the {\em preordering} on $\Ideals(M,I,m-N)$ {\em does} depend on all three of the variables $M$,$I$, and $m-N$.

Let $\pi_0\Ideals(M,I,m-N)$ be the partially-ordered set of equivalence classes in the preordered set $\Ideals(M,I,m-N)$. We aim to show that every ascending sequence in $\pi_0\Ideals(M,I,m-N)$ stabilizes. Given a pair of ideals $J_1,J_2\in \Ideals(M,I,m-N)$ such that $J_1\leq J_2$, it is routine to check that the sum $J_1+J_2$ of the ideals satisfies both $J_2\leq J_1+J_2$ and $J_1+J_2\leq J_2$, i.e., we have an equivalence $J_1+J_2\sim J_2$ in $\Ideals(M,I,m-N)$. Consequently, given an ascending sequence
\begin{equation}\label{seq 11} J_1 \leq J_2 \leq J_3 \leq J_4 \leq \dots\end{equation} in $\Ideals(M,I,m-N)$, we have a sequence of containments of ideals
\[ J_1\subseteq J_1+ J_2 \subseteq J_1+J_2+J_3\subseteq J_1+J_2+J_3+J_4\subseteq\dots ,\]
each of which is in $\Ideals(M,I,m-N)$, and which stabilizes if and only if \eqref{seq 11} stabilizes.

So we suppose that we have a sequence of containments of ideals as in \eqref{seq 231}, each of which is an element of $\Ideals(M,I,m-N)$. By the assumption of condition 3, the sequence \eqref{seq 231} is $M$-annihilator stable for each degree in $m-N$. Consequently \eqref{seq 231} stabilizes in $\pi_0\Ideals(M,I,m-N)$, as desired.

Since the ascending chains in $\pi_0\Ideals(M,I,m-N)$ all stabilize, Zorn's Lemma ensures that $\pi_0\Ideals(M,I,m-N)$ has a maximal element. Let $I^{\prime}$ be an element of $\Ideals(M,I,m-N)$ representing a maximal element of $\pi_0\Ideals(M,I,m-N)$. If $x$ is a homogeneous element of $I$, then $I^{\prime} + Rx$ is a member of $\Ideals(M,I,m-N)$ containing $I^{\prime}$. By maximality of $I^{\prime}$, we must have $((I^{\prime} + Rx)^{\perp})^d = ((I^{\prime})^{\perp})^d$ for all $d\in m-N$. Since $x\in I^{\prime} + Rx$, we must have $xm=0$ for all $m\in \left( I^{\prime} + Rx\right)^{\perp}$, hence $xm=0$ for all $m\in ((I^{\prime})^{\perp})^d$.

This argument applies for all homogeneous $x\in I$, so we have $(I^{\perp})^d\supseteq ((I^{\prime})^{\perp})^d$. The reverse containment $I^{\perp}\subseteq (I^{\prime})^{\perp}$ follows from $I^{\prime}$ being a subideal of $I$, so we have $(I^{\perp})^d = ((I^{\prime})^{\perp})^d$ for all $d\in m-N$. So $I^{\prime}$ is the desired finitely generated subideal of $I$.

Now we use the gr-injectivity of $M$ together with the graded Baer criterion. Given a graded left $R$-module homomorphism $f: \Sigma^m I \rightarrow \coprod_{s\in S} \Sigma^{d(s)} M$ with each $d(s)$ in $N$, we have a commutative diagram of graded left $R$-modules given by the solid arrows depicted below:
\begin{equation}\label{comm diag 32109}\xymatrix{
\Sigma^m I^{\prime}\ar[r] 
 & \Sigma^m I \ar[d]^f \ar[r]
 & \Sigma^m R \ar[d]^{\tilde{f}}\ar@{-->}[dl]\\
 & \coprod_{s\in S} \Sigma^{d(s)} M \ar[r]
 & \prod_{s\in S} \Sigma^{d(s)} M. }\end{equation}
In \eqref{comm diag 32109}, the map $\tilde{f}$ is obtained using the universal property of the gr-injective module $\prod_{s\in S} \Sigma^{d(s)} M$, and $I^{\prime}$ is a finitely generated homogeneous left ideal of $R$ of the kind just constructed using Zorn's Lemma, i.e., $I^{\prime}$ is contained in $I$, and $(I^{\prime})^{\perp}$ coincides with $I^{\perp}$ in every degree in $m-N$. The horizontal maps in \eqref{comm diag 32109} are each the natural subset inclusion maps.

Since $I^{\prime}$ is finitely generated, we may choose a finite homogeneous generating set $r_1, \dots ,r_k$ for it. The images $\tilde{f}(r_1), \dots ,\tilde{f}(r_k)$ of $r_1, \dots ,r_k$ land in the direct sum $\coprod_{s\in S} \Sigma^{d(s)} M$, so there is a finite subset $T$ of $S$ such that the image of $\tilde{f}\mid_{I^{\prime}}$ is contained in $\coprod_{s\in T}\Sigma^{d(s)}M\subseteq \prod_{s\in S}\Sigma^{d(s)}M$. 

Given an element $z\in \prod_{s\in S}\Sigma^{d(s)}M$ and an element $s\in S$, write $z_s$ for the component of $z$ in the factor $\Sigma^{d(s)}M$ of $\prod_{s\in S}\Sigma^{d(s)}M$.
Let $\hat{f}: \Sigma^m R\rightarrow \coprod_{s\in S}\Sigma^{d(s)}M$ be the graded left $R$-module map determined by
\begin{align*}
 \hat{f}(1)_s &= \left\{\begin{array}{ll} \tilde{f}(1)_s &\mbox{\ if\ } s\in T \\ 0 &\mbox{\ if\ } s\notin T.\end{array}\right.
\end{align*}

We have defined $\hat{f}$ so that $r_i\cdot \hat{f}(1) = r_i\cdot \tilde{f}(1)$ is true for all $i=1, \dots ,k$. Since $r_1, \dots,r_k$ generate $I^{\prime}$, we have $(\hat{f}(1) - \tilde{f}(1))_s\in (I^{\prime})^{\perp}$ for each $s\in S$. In particular, $(\hat{f}(1) - \tilde{f}(1))_s$ is in degree $m-d(s)$ in $(I^{\prime})^{\perp}\subseteq M$. Since $m-d(s)\in m-N$, we have $(\hat{f}(1) - \tilde{f}(1))_s\in I^{\perp}$ as well. Hence $r\cdot \hat{f} = r\cdot \tilde{f}(1)$ for all $r\in I$, i.e., $\hat{f}$ and $\tilde{f}$ agree on $I$. Consequently $\hat{f}$ fills in the dotted arrow in diagram \eqref{comm diag 32109} and makes the upper-left triangle commute. This is precisely the condition necessary to obtain gr-injectivity of $\coprod_{s\in S}\Sigma^{d(s)}M$ from the graded Baer criterion. Hence $M$ is $(\Sigma,m-N)$-injective.
\item[1 implies 4, and 2 implies 5] The suspension functor $\Sigma$ is an automorphism of the graded module category, so a graded module is $(\Sigma,N)$-injective if and only if it is $(\Sigma,m+N)$-injective for all integers $m$.
\end{description}
\end{proof}

\begin{corollary}\label{simple cor}
Let $R$ be a graded ring, and let $M$ be a gr-injective graded left $R$-module. Then:
\begin{itemize}
\item $M$ is strictly $\Sigma$-injective if and only if, for every integer $d$ and every ascending chain \begin{equation}\label{seq 232} I_0\subseteq I_1 \subseteq I_2 \subseteq \dots\end{equation} of homogeneous left ideals of $R$, the descending chain of graded submodules of $M$
\begin{equation}\label{seq 233} I_0^{\perp}\supseteq I_1^{\perp} \supseteq I_2^{\perp} \supseteq \dots\end{equation} is eventually constant in degree $d$.
\item $M$ is unboundedly $\Sigma$-injective if and only if, for every ascending chain \eqref{seq 232} %\begin{equation*}%\label{seq 234} 
%I_0\subseteq I_1 \subseteq I_2 \subseteq \dots\end{equation*} 
of homogeneous left ideals of $R$, the descending chain \eqref{seq 233} of graded submodules of $M$
%\begin{equation}\label{seq 235} I_0^{\perp}\supseteq I_1^{\perp} \supseteq I_2^{\perp} \supseteq \dots\end{equation} 
is eventually constant\footnote{To be clear: the difference between this condition and the previous condition is that, for unbounded $\Sigma$-injectivity, we must have a single integer $\ell$ such that the sequence \eqref{seq 233} is constant, in all degrees $d$, starting with the $\ell$th stage in the sequence. By contrast, for {\em strict} $\Sigma$-injectivity, the sequence \eqref{seq 233} could stabilize at a different stage for each given degree, without there being a single stage by which \eqref{seq 233} is constant in {\em every} degree.}.
\item $M$ is bounded-abovely $\Sigma$-injective (respectively, bounded-belowly $\Sigma$-injective) if and only if, for every integer $m$ and every ascending chain \eqref{seq 232} %\begin{equation*}%\label{seq 234} I_0\subseteq I_1 \subseteq I_2 \subseteq \dots\end{equation*} 
of homogeneous left ideals of $R$, the descending chain \eqref{seq 233} of graded submodules of $M$
%\begin{equation*}%\label{seq 235}  I_0^{\perp}\supseteq I_1^{\perp} \supseteq I_2^{\perp} \supseteq \dots\end{equation*} 
is eventually constant in all degrees $\geq m$ (respectively, all degrees $\leq m$). 
\end{itemize}
\end{corollary}

\begin{corollary}\label{simpler cor}
Let $R$ be a graded ring. Every bounded-below gr-injective graded $R$-module is bounded-belowly $\Sigma$-injective. Furthermore, every bounded-above gr-injective graded $R$-module is bounded-abovely $\Sigma$-injective.
\end{corollary}

\section{Injective graded $A^*$-comodules and injective graded $A$-modules.}
\label{Injective graded...}

One special case of Corollary \ref{simpler cor} is the theorem of Margolis (see sections 13.2 and 13.3 of \cite{MR738973}): a uniformly bounded-below direct sum of bounded-below gr-injective modules over a $P$-algebra (such as the Steenrod algebra) is gr-injective. Corollary \ref{simpler cor} establishes that this is in fact a general feature of gr-self-injective rings, and does not require the ground ring to be a $P$-algebra.

Another special case of interest is the linear dual of the previous case. Let $A$ be the mod $p$ Steenrod algebra, and let $dA$ denote the $\mathbb{F}_p$-linear dual graded $A$-module of $A$, regarded as a graded left $A$-module via the adjoint action\footnote{This is a standard construction; see e.g. section 11.3 of \cite{MR738973}, or see the discussion of the adjoint action in this paper, belo.w} of $A$. It is easy to show that $dA$ is gr-injective: this was Margolis's example of a demonstrably non-free gr-injective $A$-module, as in Proposition 12 in section 11.3 of \cite{MR738973}. We can now address the question of whether direct sums of copies of $dA$ are also gr-injective:
\begin{prop}\label{da and sigma-inj}
The graded $A$-module $dA$ is bounded-abovely $\Sigma$-injective, but $dA$ is not bounded-belowly $\Sigma$-injective.
\end{prop}
\begin{proof}
Since $dA$ is bounded above, it is a special case of Corollary \ref{simpler cor} that $dA$ is bounded-abovely $\Sigma$-injective. So all that remains is to show that $dA$ is not bounded-belowly $\Sigma$-injective.

Consider the ascending chain of homogeneous left ideals
\[ A(\Sq^1) \subseteq A(\Sq^1,\Sq^2)\subseteq A(\Sq^1,\Sq^2,\Sq^4) \subseteq A(\Sq^1,\Sq^2,\Sq^4,\Sq^8) \subseteq 
 \dots \]
of $A$. (Here we work with the mod $2$ Steenrod algebra for convenience of exposition, but an analogous argument works at odd primes.) 

We consider the graded annihilator submodules of $dA$ for each of these ideals. The annihilator submodule $(A(\Sq^1,\Sq^2,\dots ,\Sq^{2^n}))^{\perp}\subseteq dA$ depends only on the action of $\Sq^1,\Sq^2,\dots,\Sq^{2^n}$ on $dA$, i.e., the structure of $dA$ as a graded module over the subalgebra $A(n)$ of $A$ generated by $\Sq^1,\Sq^2,\dots,\Sq^{2^n}$. It follows from generalities about $P$-algebras (see Theorem 12 in section 13.3 of \cite{MR738973}) that, for each $n$, the underlying $A(n)$-module of $A$ is free on a set of homogeneous generators whose degrees are bounded below by zero, but are not bounded above. 

Since $A(n)$ is a Frobenius algebra, the $\mathbb{F}_2$-linear dual $dA$ of $A$ is also free as an $A(n)$-module, but it is free on a set of homogeneous generators whose degrees are bounded above by zero, but are not bounded below. The upshot is that, for each integer $m$, there exists an integer $\ell$ such that the sequence 
\begin{equation}\label{seq 8857435} 
\left(A(\Sq^1)\right)^{\perp} \supseteq \left(A(\Sq^1,\Sq^2)\right)^{\perp}\supseteq \left(A(\Sq^1,\Sq^2,\Sq^4)\right)^{\perp} \supseteq \left(A(\Sq^1,\Sq^2,\Sq^4,\Sq^8)\right)^{\perp} \supseteq  \dots
%\left( A(\Sq^1,\Sq^2, \dots,\Sq^{2^{\ell}})\right)^{\perp}\supseteq  \left( A(\Sq^1,\Sq^2, \dots,\Sq^{2^{\ell+1}})\right)^{\perp}\supseteq \left( A(\Sq^1,\Sq^2, \dots,\Sq^{2^{\ell+2}})\right)^{\perp}\supseteq  \dots 
\end{equation}
is constant after its $\ell$th stage in all degrees {\em greater than $m$}. However, there is no integer $\ell$ such that \eqref{seq 8857435} is constant after its $\ell$th stage in all degrees {\em less than $m$.}

Consequently, by Corollary \ref{simple cor}, $dA$ is not bounded-belowly $\Sigma$-injective.
\end{proof}

There is a good reason to consider the various forms of graded $\Sigma$-injectivity for the specific $A$-module $dA$. Recall that the category $\gr\Comod(A_*)$ of graded right comodules over the dual Steenrod algebra $A_*$ admits a {\em covariant} embedding into the category $\gr\Mod(A)$ of graded left $A$-modules. This is a purely algebraic construction, quite classical: I do not know its earliest appearance, but as far as I know, \cite{MR686116} was its first mention in the context of algebraic topology. 
Given a graded right $A_*$-comodule $M$ with coaction map $\psi: M \rightarrow M\otimes_{\mathbb{F}_p} A_*$, identify $A$ with its double dual $A_{**}$, and then the action map 
\begin{align*}
 A\times M \stackrel{\cong}{\longrightarrow} A_{**}\times M  \rightarrow M \end{align*} 
sends a pair $(f,m)\in A_{**}\times M$ to the image of $m$ under the composite
\begin{equation}\label{composite 035494} M\stackrel{\psi}{\longrightarrow} M\otimes_{\mathbb{F}_p} A_* \stackrel{M\otimes f}{\longrightarrow} M\otimes_{\mathbb{F}_p}\mathbb{F}_p \stackrel{\cong}{\longrightarrow} M. \end{equation}
This action of $A$ on $M$ is called the {\em adjoint action}. 
This construction yields a covariant, exact, faithful, full functor\footnote{Our presentation of the functor $\iota$ and its properties presumes that we are using cohomological gradings throughout, so that $A$ is in nonnegative degrees and $A_*$ is in nonpositive degrees. This ensures that the adjoint action indeed respects the gradings.} $\iota: \gr\Comod(A_*)\rightarrow \gr\Mod(A)$. The book \cite{MR2012570} is an excellent reference for these and other properties of the functor $\iota$, not only for the dual Steenrod algebra $A_*$, but also for a general coalgebra in place of $A_*$.

As the covariant embedding $\iota$ of $\gr\Comod(A_*)$ into $\gr\Mod(A)$ is exact, one expects $\iota$ to preserve a great deal of cohomological information. If $\iota$ were to send injective objects to injective objects, then it would be easy to prove that right-derived functors in the category of graded $A_*$-comodules can be computed by first embedding the comodule category into graded $A$-modules, then calculating right derived functors there. Right derived functors in the comodule category arise in practice in algebraic topology\footnote{E.g. the input term of the $H\mathbb{F}_p$-Adams spectral sequence for a non-finite-type spectrum, which generally only has a description in terms of $\Cotor_{A_*}$ rather than $\Ext_A$; see chapter 2 of \cite{MR860042}. Also, the input term of the Sadofsky spectral sequence \cite{sadofsky2001homology}, which is comprised of the right derived functors of product in the category of graded $A_*$-comodules.}, but module categories are far more familiar and understood than comodule categories, it is of obvious interest to know whether homological algebra in graded $A_*$-comodules can be described in terms of homological algebra in graded $A$-modules! Yet it seems this question---of whether $\iota$ sends gr-injective comodules to gr-injective modules---is not addressed in the literature. As far as the author knows, an answer is not known.

It turns out that the graded $\Sigma$-injectivity of $dA$ is precisely the key to answering that question:
\begin{theorem}\label{failure of iota to preserve gr-injs}
Let $A$ be the mod $p$ Steenrod algebra for any prime $p$. Then the covariant embedding $\iota$ of graded $A_*$-comodules into graded $A$-modules does {\em not} send gr-injectives to gr-injectives. 

However, if $M$ is a {\em bounded above} graded $A_*$-comodule, then $\iota(M)$ {\em is} gr-injective.
\end{theorem}
\begin{proof}
The gr-injective right $A_*$-comodules are the retracts of the extended graded right $A_*$-comodules, i.e., those of the form $V\otimes_{\mathbb{F}_p}A_*$ for a graded $\mathbb{F}_p$-vector space $V$. Equivalently, the gr-injective right $A_*$-comodules are the summands of coproducts of suspensions of $A_*$. The functor $\iota$ sends $A_*$ to $dA$, and $\iota$ preserves coproducts since it is a left adjoint\footnote{The right adjoint to $\iota$ is called ``the rational functor'' or ``the trace functor'' in the coalgebra literature: see \cite{MR2012570} for a nice treatment.}. Hence Proposition \ref{da and sigma-inj} yields immediately that $\iota(M)$ is gr-injective if $M$ is a bounded-above gr-injective comodule, while $\iota(M)$ must fail to be gr-injective for some non-bounded-above gr-injective comodules.
\end{proof}

To close, here is a related question that the author is curious about, but does not know an answer to:
\begin{question}\label{dim q} 
Let $A$ be the mod $p$ Steenrod algebra for any prime $p$. Suppose that $M$ is a graded right $A_*$-comodule which is gr-injective. 
Is the gr-injective dimension of the $A$-module $\iota(M)$ at most $1$?
\end{question}
While Question \ref{dim q} is motivated by the desire to understand the homological algebra of $A_*$-comodules because of their applications in algebraic topology, the question is also quite close to one that has been considered for other reasons, as follows. Suppose that $R$ is a {\em left Noetherian} nonnegatively-graded ring. A theorem of van der Bergh (\cite{MR1469646}, see also \cite{yekutielivandenberghpreprint} for Yekutieli's elegant proof) establishes that, given a gr-injective left $R$-module $M$, the injective dimension of the underlying ungraded $R$-module of $M$ is at most one. 

Of course the Steenrod algebra is not left Noetherian! But Yekutieli remarks, in \cite{yekutielivandenberghpreprint}, that ``[w]e do not know if the noetherian condition in Theorem 1 is necessary.'' If the theorem of van der Bergh and Yekutieli can be proven without the Noetherian hypothesis, then the injective dimension of a coproduct of copies of $dA$ is at most $1$, and consequently, %(since the injective dimension is an upper bound on the gr-injective dimension) 
for any gr-injective $A_*$-comodule $M$, the gr-injective dimension of $\iota(M)$ could not exceed $1$, yielding an affirmative answer to Question \ref{dim q}.

\bibliography{/home/asalch/texmf/tex/salch}{}
\bibliographystyle{plain}
\end{document}